\documentclass{amsart}
\usepackage{amsmath}
\usepackage{amssymb}
\usepackage{amsfonts}

\newtheorem{theorem}{Theorem}
\theoremstyle{plain}

\newtheorem{definition}{Definition}

\newtheorem{lemma}{Lemma}

\newtheorem{remark}{Remark}

\numberwithin{equation}{section}
\input{tcilatex}

\begin{document}
\title[On new inequalities]{New fractional inequalities of Ostrowski-Gr\"{u}%
ss type }
\author{Mehmet Zeki Sarikaya$^{\star }$}
\address{Department of Mathematics, Faculty of Science and Arts, D\"{u}zce
University, D\"{u}zce, Turkey}
\email{sarikayamz@gmail.com}
\thanks{$^{\star }$corresponding author}
\author{Hatice YALDIZ}
\email{yaldizhatice@gmail.com}
\author{Nagihan Basak}
\email{nagihan.basak@hotmail.com}

\begin{abstract}
In this paper, we improve and further generalize some Ostrowski-Gr\"{u}ss
type inequalities for the fractional integrals by using new Montogomery
identities.
\end{abstract}

\maketitle

\section{Introduction and preliminary results}

The inequality of Ostrowski gives us an estimate for the deviation of the
values of a smooth function from its mean value. More precisely, if $%
f:[a,b]\rightarrow \mathbb{R}$ is a differentiable function with bounded
derivative, then%
\begin{equation*}
\left\vert f(x)-\frac{1}{b-a}\int\limits_{a}^{b}f(t)dt\right\vert \leq \left[
\frac{1}{4}+\frac{(x-\frac{a+b}{2})^{2}}{(b-a)^{2}}\right] (b-a)\left\Vert
f^{\prime }\right\Vert _{\infty }
\end{equation*}%
for every $x\in \lbrack a,b]$. Moreover the constant $1/4$ is the best
possible.

For some generalizations of this classic fact see the book \cite[p.468-484]%
{[13]} by Mitrinovic, Pecaric and Fink. A simple proof of this fact can be \
done by using the following identity \cite{[13]}:

If $f:[a,b]\rightarrow \mathbb{R}$ is differentiable on $[a,b]$ with the
first derivative $f^{\prime }$ integrable on $[a,b],$ then Montgomery
identity holds:%
\begin{equation*}
f(x)=\frac{1}{b-a}\int\limits_{a}^{b}f(t)dt+\int\limits_{a}^{b}P_{1}(x,t)f^{%
\prime }(t)dt,
\end{equation*}%
where $P_{1}(x,t)$ is the Peano kernel defined by%
\begin{equation*}
P_{1}(x,t):=\left\{ 
\begin{array}{ll}
\dfrac{t-a}{b-a}, & a\leq t<x \\ 
&  \\ 
\dfrac{t-b}{b-a}, & x\leq t\leq b.%
\end{array}%
\right.
\end{equation*}

This inequality provides an upper bound for the approximation of integral
mean of a function $f$ by the functional value $f(x)$ at $x\in \lbrack a,b].$
In 2001, , Cheng \cite{[1]} proved the following Ostrowski-Gr\"{u}ss type
integral inequality.

\begin{theorem}
\label{t1} Let $I\subset 
\mathbb{R}
$ be an open interval, $a,b\in I,a<b$. If $f:I\rightarrow 
\mathbb{R}
$ is a differentiable function such that there exist constants $\gamma
,\Gamma \in 
\mathbb{R}
$, with $\gamma \leq f^{^{\prime }}\left( x\right) \leq \Gamma $, $x\in %
\left[ a,b\right] $. Then have%
\begin{eqnarray}
&&\left\vert \frac{1}{2}f\left( x\right) -\frac{\left( x-b\right) f\left(
b\right) -\left( x-a\right) f\left( a\right) }{2\left( b-a\right) }-\frac{1}{%
b-a}\int\limits_{a}^{b}f\left( t\right) dt\right\vert  \label{hh} \\
&&  \notag \\
&\leq &\frac{\left( x-a\right) ^{2}+\left( b-x\right) ^{2}}{8\left(
b-a\right) }\left( \Gamma -\gamma \right) \text{, for all }x\in \left[ a,b%
\right] .  \notag
\end{eqnarray}
\end{theorem}

Theorem \ref{t1} is a generalization of the following Ostrowski-Gr\"{u}ss
type integral inequality, which was firstly by Dragomir and Wang in \cite%
{[2]} and further improved by Matic et al. in \cite{[5]}.

\begin{theorem}
\label{t2} Let the assumptions of Theorem \ref{t1} hold. Then for all $x\in %
\left[ a,b\right] $, we have
\end{theorem}

\begin{eqnarray}
&&\left\vert f\left( x\right) -\frac{f\left( b\right) -f\left( a\right) }{b-a%
}\left( x-\frac{a+b}{2}\right) -\frac{1}{b-a}\int\limits_{a}^{b}f\left(
t\right) dt\right\vert  \label{hhh} \\
&&  \notag \\
&\leq &\frac{1}{4}\left( b-a\right) \left( \Gamma -\gamma \right) .  \notag
\end{eqnarray}

The above two inequalities are both connections between the Ostrowski
inequality \cite{[3]}\ and the Gr\"{u}ss inequality \cite{[4]}\ and can be
applied to bound some special mean and some numerical quadrature rules.

During the past few years many researchers have given considerable attention
to the above inequalities and various generalizations, extensions and
variants of these inequalities have appeared in the literature, see \cite%
{[1]}, \cite{[2]}, \cite{[5]}, \cite{[24]} and the references cited therein.
For recent results and generalizations concerning Ostrowski and Gr\"{u}ss
inequalities, we refer the reader to the recent papers \cite{[1]}-\cite{[5]}%
, \cite{[13]}-\cite{[16]}, \cite{[18]}-\cite{[24]}.

The theory of fractional calculus has known an intensive development over
the last few decades. It is shown that derivatives and\ integrals of
fractional type provide an adequate mathematical modelling of real objects\
and processes see (\cite{[6]}, \cite{[7]}, \cite{[9]}, \cite{[10]}, \cite%
{[21]}, \cite{[22]}). Therefore, the study of fractional differential
equations need more developmental of inequalities of fractional type. The
main aim of this work is to develop new weighted Montgomery identity for
Riemann-Liouville fractional integrals that will be used to establish new
weighted Ostrowski inequalities. Let us begin by introducing this type of
inequality.

In \cite{[6]} and \cite{[21]}, the authors established some inequalities for
differentiable mappings which are connected with Ostrowski type inequality
by used the Riemann-Liouville fractional integrals, and they used the
following lemma to prove their results:

\begin{lemma}
\label{l} Let $f:I\subset \mathbb{R}\rightarrow \mathbb{R}$ be
differentiable function on $I^{\circ }$ with $a,b\in I$ ($a<b$) and $%
f^{\prime }\in L_{1}[a,b]$, then%
\begin{equation}
f(x)=\frac{\Gamma (\alpha )}{b-a}(b-x)^{1-\alpha }{\Large J}_{a}^{\alpha
}f(b)-{\Large J}_{a}^{\alpha -1}(P_{2}(x,b)f(b))+{\Large J}_{a}^{\alpha
}(P_{2}(x,b)f^{^{\prime }}(b)),\ \ \ \alpha \geq 1,  \label{z}
\end{equation}%
where $P_{2}(x,t)$ is the fractional Peano kernel defined by%
\begin{equation*}
P_{2}(x,t)=\left\{ 
\begin{array}{ll}
\dfrac{t-a}{b-a}(b-x)^{1-\alpha }\Gamma (\alpha ), & a\leq t<x \\ 
&  \\ 
\dfrac{t-b}{b-a}(b-x)^{1-\alpha }\Gamma (\alpha ), & x\leq t\leq b.%
\end{array}%
\right.
\end{equation*}
\end{lemma}

In this article, we use the Riemann-Liouville fractional integrals to
establish some new integral inequalities of Ostrowski-Gr\"{u}ss type. From
our results, the classical Ostrowski-Gr\"{u}ss type inequalities can be
deduced as some special cases.

Firstly, we give some necessary definitions and mathematical preliminaries
of fractional calculus theory which are used further in this paper. More
details, one can consult \cite{[12]}, \cite{[17]}.

\begin{definition}
The Riemann-Liouville fractional integral operator of order $\alpha \geq 0$
with $a\geq 0$ is defined as%
\begin{eqnarray*}
J_{a}^{\alpha }f(x) &=&\frac{1}{\Gamma (\alpha )}\dint\limits_{a}^{x}(x-t)^{%
\alpha -1}f(t)dt, \\
J_{a}^{0}f(x) &=&f(x).
\end{eqnarray*}
\end{definition}

Recently, many authors have studied a number of inequalities by used the
Riemann-Liouville fractional integrals, see (\cite{[6]}, \cite{[7]}, \cite%
{[9]}, \cite{[10]}, \cite{[21]}, \cite{[22]}) and the references cited
therein.

\section{Main Results}

\begin{lemma}
\label{lm} Let $f:I\subset \mathbb{R}\rightarrow \mathbb{R}$ be a
differentiable function on $I^{\circ }$ with $a,b\in I$ ($a<b$)$,$ $\alpha
\geq 1$ and $f^{\prime }\in L_{1}[a,b]$, then the generalization of the
Montgomery identity for fractional integrals holds:%
\begin{eqnarray}
f\left( x\right) &=&(\alpha +1)\Gamma (\alpha )\frac{\left( b-x\right)
^{1-\alpha }}{(b-a)}{\Large J}_{a}^{\alpha }f(b)-{\Large J}_{a}^{\alpha
-1}(P_{2}(x,b)f(b))-\frac{\left( b-x\right) ^{2-\alpha }}{(b-a)}\Gamma
(\alpha )J_{a}^{\alpha -1}f\left( b\right)  \notag \\
&&  \label{9} \\
&&-\frac{\left( b-x\right) ^{1-\alpha }(x-a)}{\left( b-a\right) ^{2-\alpha }}%
f(a)+2J_{a}^{\alpha }\left( K_{1}\left( x,b\right) f^{^{\prime }}\left(
b\right) \right) ,  \notag
\end{eqnarray}%
where $K_{1}\left( x,t\right) $ is the fractional Peano kernel defined by%
\begin{equation}
K_{1}\left( x,t\right) :=\left\{ 
\begin{array}{ll}
\left( t-\dfrac{a+x}{2}\right) \dfrac{\left( b-x\right) ^{1-\alpha }}{b-a}%
\Gamma \left( \alpha \right) , & t\in \lbrack a,x) \\ 
\left( t-\dfrac{b+x}{2}\right) \dfrac{\left( b-x\right) ^{1-\alpha }}{b-a}%
\Gamma \left( \alpha \right) , & t\in \lbrack x,b].%
\end{array}%
\right.  \label{91}
\end{equation}
\end{lemma}

\begin{proof}
By definition of $K_{1}\left( x,t\right) $, we have%
\begin{eqnarray*}
&&J_{a}^{\alpha }\left( K_{1}\left( x,b\right) f^{^{\prime }}\left( b\right)
\right) \\
&& \\
&=&\frac{1}{\Gamma \left( \alpha \right) }\dint\limits_{a}^{b}\left(
b-t\right) ^{\alpha -1}K_{1}\left( x,t\right) f^{\prime }\left( t\right) dt
\\
&& \\
&=&\frac{\left( b-x\right) ^{1-\alpha }}{b-a}\left[ \dint\limits_{a}^{x}%
\left( b-t\right) ^{\alpha -1}\left( t-\frac{a+x}{2}\right) f^{^{\prime
}}\left( t\right) dt+\dint\limits_{x}^{b}\left( b-t\right) ^{\alpha
-1}\left( t-\frac{b+x}{2}\right) f^{^{\prime }}\left( t\right) dt\right]
\end{eqnarray*}%
that can be written as%
\begin{equation}
J_{a}^{\alpha }\left( K_{1}\left( x,b\right) f^{^{\prime }}\left( b\right)
\right) =\frac{1}{2}{\Large J}_{a}^{\alpha }(P_{2}(x,b)f^{^{\prime }}(b))+%
\frac{\left( b-x\right) ^{1-\alpha }}{2(b-a)}\dint\limits_{a}^{b}\left(
b-t\right) ^{\alpha -1}\left( t-x\right) f^{^{\prime }}\left( t\right) dt.
\label{10}
\end{equation}%
For term in the right hand side of (\ref{10}) integrating by parts implies
that 
\begin{eqnarray}
&&\dint\limits_{a}^{b}\left( b-t\right) ^{\alpha -1}\left( t-x\right)
f^{^{\prime }}\left( t\right) dt  \notag \\
&=&\left( b-x\right) \dint\limits_{a}^{b}\left( b-t\right) ^{\alpha
-1}f^{^{\prime }}\left( t\right) dt-\dint\limits_{a}^{b}\left( b-t\right)
^{\alpha }f^{^{\prime }}\left( t\right) dt  \label{11} \\
&&  \notag \\
&=&(x-a)\left( b-a\right) ^{\alpha -1}f(a)+\left( b-x\right) \Gamma (\alpha
)J_{a}^{\alpha -1}f\left( b\right) -\Gamma (\alpha +1)J_{a}^{\alpha }f\left(
b\right)  \notag
\end{eqnarray}

Substituting ${\Large J}_{a}^{\alpha }(P_{2}(x,b)f^{^{\prime }}(b))$ in the
Lemma \ref{l} and (\ref{11}) in (\ref{10}), we obtain that%
\begin{eqnarray*}
&&J_{a}^{\alpha }\left( K_{1}\left( x,b\right) f^{^{\prime }}\left( b\right)
\right) \\
&& \\
&=&\frac{1}{2}f(x)-(\alpha +1)\Gamma (\alpha )\frac{\left( b-x\right)
^{1-\alpha }}{2(b-a)}{\Large J}_{a}^{\alpha }f(b)+\frac{1}{2}{\Large J}%
_{a}^{\alpha -1}(P_{2}(x,b)f(b)) \\
&& \\
&&+\frac{\left( b-x\right) ^{1-\alpha }(x-a)}{2}\left( b-a\right) ^{\alpha
-2}f(a)+\frac{\left( b-x\right) ^{2-\alpha }}{2(b-a)}\Gamma (\alpha
)J_{a}^{\alpha -1}f\left( b\right) .
\end{eqnarray*}%
The proof is completed.
\end{proof}

\begin{remark}
Letting $\alpha =1,$ formula (\ref{9}) reduces the following identity%
\begin{equation*}
\frac{1}{2}f\left( x\right) =\frac{1}{(b-a)}\dint\limits_{a}^{b}f(t)dt+\frac{%
\left( x-b\right) f\left( b\right) -(x-a)f(a)}{2(b-a)}+\dint\limits_{a}^{b}K%
\left( x,t\right) f^{^{\prime }}(t)dt
\end{equation*}%
where%
\begin{equation*}
K\left( x,t\right) :=\left\{ 
\begin{array}{ll}
\left( t-\dfrac{a+x}{2}\right) , & t\in \lbrack a,x) \\ 
\left( t-\dfrac{b+x}{2}\right) , & t\in \lbrack x,b]%
\end{array}%
\right.
\end{equation*}%
which was given by Tong and Guan in \cite{[24]}. Using the above identity,
the authors proved another simple proof of Theorem \ref{t1}.
\end{remark}

Now using the new Montgomery identity for fractional integrals (\ref{9}) and
the corresponding fractional Peano kernel (\ref{91}), we derive a new
Ostrowski-Gr\"{u}ss type inequality of fractional type.

\begin{theorem}
Let\ $f$ be a differentiable \ function on $\left[ a,b\right] $ and $%
\left\vert f^{\prime }\left( x\right) \right\vert \leq M$ for any $x\in %
\left[ a,b\right] $. Then the following fractional inequality holds:%
\begin{eqnarray}
&&\left\vert \frac{1}{2}f\left( x\right) -(\alpha +1)\Gamma (\alpha )\frac{%
\left( b-x\right) ^{1-\alpha }}{2(b-a)}{\Large J}_{a}^{\alpha }f(b)+\frac{1}{%
2}{\Large J}_{a}^{\alpha -1}(P_{2}(x,b)f(b))\right.  \notag \\
&&  \label{s} \\
&&\left. +\frac{\left( b-x\right) ^{2-\alpha }}{2(b-a)}\Gamma (\alpha
)J_{a}^{\alpha -1}f\left( b\right) +\frac{\left( b-x\right) ^{1-\alpha }(x-a)%
}{2\left( b-a\right) ^{2-\alpha }}f(a)\right\vert  \notag \\
&&  \notag \\
&\leq &\frac{M\left( b-x\right) ^{1-\alpha }}{\left( b-a\right) }\left[ 
\frac{\left( b-a\right) ^{\alpha }(x-a)+\left( b-x\right) ^{\alpha }(a+b-2x)%
}{2\alpha }\right]  \notag
\end{eqnarray}%
for $\alpha \geq 1.$
\end{theorem}

\begin{proof}
From Lemma \ref{lm}, we have%
\begin{eqnarray*}
&&\left\vert \frac{1}{2}f\left( x\right) -(\alpha +1)\Gamma (\alpha )\frac{%
\left( b-x\right) ^{1-\alpha }}{2(b-a)}{\Large J}_{a}^{\alpha }f(b)+\frac{1}{%
2}{\Large J}_{a}^{\alpha -1}(P_{2}(x,b)f(b))\right. \\
&& \\
&&\left. +\frac{\left( b-x\right) ^{2-\alpha }}{2(b-a)}\Gamma (\alpha
)J_{a}^{\alpha -1}f\left( b\right) +\frac{\left( b-x\right) ^{1-\alpha }(x-a)%
}{2\left( b-a\right) ^{2-\alpha }}f(a)\right\vert \\
&& \\
&=&\frac{1}{\Gamma \left( \alpha \right) }\left\vert \int_{a}^{b}\left(
b-t\right) ^{\alpha -1}K_{1}\left( x,t\right) f^{\prime }\left( t\right)
dt\right\vert .
\end{eqnarray*}%
Taking into account the assumptions on the function $f$, it yields%
\begin{eqnarray}
&&\left\vert \frac{1}{2}f\left( x\right) -(\alpha +1)\Gamma (\alpha )\frac{%
\left( b-x\right) ^{1-\alpha }}{2(b-a)}{\Large J}_{a}^{\alpha }f(b)+\frac{1}{%
2}{\Large J}_{a}^{\alpha -1}(P_{2}(x,b)f(b))\right.  \notag \\
&&  \notag \\
&&\left. +\frac{\left( b-x\right) ^{2-\alpha }}{2(b-a)}\Gamma (\alpha
)J_{a}^{\alpha -1}f\left( b\right) +\frac{\left( b-x\right) ^{1-\alpha }(x-a)%
}{2\left( b-a\right) ^{2-\alpha }}f(a)\right\vert  \notag \\
&&  \label{14} \\
&\leq &\frac{M\left( b-x\right) ^{1-\alpha }}{\left( b-a\right) }\left[
\dint\limits_{a}^{x}\left( b-t\right) ^{\alpha -1}\left\vert t-\frac{a+x}{2}%
\right\vert dt\right.  \notag \\
&&  \notag \\
&&\left. +\dint\limits_{x}^{b}\left( b-t\right) ^{\alpha -1}\left\vert t-%
\frac{b+x}{2}\right\vert dt\right] .  \notag
\end{eqnarray}%
Noting the left hand side of (\ref{14}) by $I$ then integrating by parts the
right hand side of (\ref{14}), we obtain%
\begin{equation*}
I\leq \frac{M\left( b-x\right) ^{1-\alpha }}{2\left( b-a\right) }\left[ 
\frac{\left( b-a\right) ^{\alpha }-\left( b-x\right) ^{\alpha }}{\alpha }%
\left( x-a\right) +\frac{\left( b-x\right) ^{\alpha +1}}{\alpha }\right] .
\end{equation*}%
Consequently%
\begin{equation*}
I\leq \frac{M\left( b-x\right) ^{1-\alpha }}{\left( b-a\right) }\left[ \frac{%
\left( b-a\right) ^{\alpha }(x-a)+\left( b-x\right) ^{\alpha }(a+b-2x)}{%
2\alpha }\right] .
\end{equation*}%
This completes the proof.
\end{proof}

\begin{remark}
Letting $\alpha =1,$ formula (\ref{s}) reduces the following inequality 
\begin{eqnarray}
&&\left\vert \frac{1}{2}f\left( x\right) -\frac{\left( x-b\right) f\left(
b\right) -(x-a)f(a)}{2(b-a)}-\frac{1}{(b-a)}\dint\limits_{a}^{b}f(t)dt\right%
\vert  \notag \\
&&  \notag \\
&\leq &\frac{M}{\left( b-a\right) }\left[ \frac{(x-a)^{2}+\left( b-x\right)
^{2}}{2}\right] .  \notag
\end{eqnarray}%
which connected with Ostrowski-Gr\"{u}ss type integral inequality. If we
take $x=\frac{a+b}{2}$ in this inequality, it follows that%
\begin{eqnarray}
&&\left\vert \frac{1}{2}f\left( \frac{a+b}{2}\right) +\frac{f\left( b\right)
+f(a)}{4}-\frac{1}{(b-a)}\dint\limits_{a}^{b}f(t)dt\right\vert  \notag \\
&&  \notag \\
&\leq &\frac{M\left( b-a\right) }{4}.  \notag
\end{eqnarray}
\end{remark}


\begin{thebibliography}{99}
\bibitem{[1]} X. L. Cheng, \textit{Improvement of some Ostrowski-Gr\"{u}ss
type inequalities}, Computers Math. Applic, 42 (2001), 109-114.

\bibitem{[2]} S. S. Dragomir and S. Wang,\textit{\ An inequality of
Ostrowski-Gr\"{u}ss type and its applications to the estimation of error
bounds for some special means and for home numerical quadrature rules},
Computers Math. Applic,33(11)(1997),15-20.

\bibitem{[8]} P. Cerone and S.S. Dragomir, \textit{Trapezoidal type rules
from an inequalities point of view}, Handbook of Analytic-Computational
Methods in Applied Mathematics, CRC Press N.Y. (2000).

\bibitem{[3]} D. S. Mitrinovic, J. Pecaric and A. M. Fink,\textit{\
Inequalities for Functions and Their Integrals and Derivatives} Kluwer
Academic, Dordrecht, (1994).

\bibitem{[4]} D. S. Mitrinovic, J. Pecaric and A. M. Fink,\textit{Classical
and New Inequalities in Analysis}, Kluwer Academic, Dordrecht, (1993).

\bibitem{[5]} M. Matic, J. Pecaric and N. Ujevic, \textit{Improvement and
further generalization of inequalities of Ostrowski-Gr\"{u}ss type,}%
Computers Math. Appl., 39(3/4), (2000), 161-175.

\bibitem{[6]} G. Anastassiou, M.R. Hooshmandasl, A. Ghasemi and F.
Moftakharzadeh, \textit{Montgomery identities for fractional integrals and
related fractional inequalities}, J. Inequal. in Pure and Appl. Math, 10(4),
2009, Art. 97, 6 pp.

\bibitem{[7]} S. Belarbi and Z. Dahmani, \textit{On some new fractional
integral inequalities}, J. Inequal. in Pure and Appl. Math, 10(3), 2009,
Art. 97, 6 pp.

\bibitem{[9]} Z. Dahmani, L. Tabharit and S. Taf, \textit{Some fractional
integral inequalities}, Nonlinear Science Letters A, 2(1), 2010, p.155-160.

\bibitem{[10]} Z. Dahmani, L. Tabharit and S. Taf, \textit{New inequalities
via Riemann-Liouville fractional integration}, J. Advance Research Sci.
Comput., 2(1), 2010, p.40-45.

\bibitem{[12]} R. Gorenflo and F. Mainardi, \textit{Fractionalcalculus:
integral and differentiable equations of fractional order}, Springer Verlag,
Wien, 1997, p.223-276.

\bibitem{[13]} D. S. Mitrinovic, J. E. Pecaric and A. M. Fink, \textit{%
Inequalities involving functions and their integrals and derivatives,}
Kluwer Academic Publishers, Dordrecht, 1991.

\bibitem{[14]} S.S. Dragomir and N. S. Barnett, \textit{An Ostrowski type
inequality for mappings whose second derivatives are bounded and applications%
}, RGMIA Research Report Collection, V.U.T., 1(1999), 67-76.

\bibitem{[15]} S.S. Dragomir, \textit{An Ostrowski type inequality for
convex functions}, Univ. Beograd. Publ. Elektrotehn. Fak. Ser. Mat. 16
(2005), 12--25.

\bibitem{[16]} Z. Liu, \textit{Some companions of an Ostrowski type
inequality and application}, J. Inequal. in Pure and Appl. Math, 10(2),
2009, Art. 52, 12 pp.

\bibitem{[17]} S. G. Samko, A. A Kilbas, O. I. Marichev, \textit{\
Fractional Integrals and Derivatives Theory and Application}, Gordan and
Breach Science, New York, 1993.

\bibitem{[18]} M. Z. Sarikaya, \textit{On the Ostrowski type integral
inequality}, Acta Math. Univ. Comenianae, Vol. LXXIX, 1(2010), pp. 129-134.

\bibitem{[19]} M. Z. Sarikaya, \textit{On the Ostrowski type integral
inequality for double integrals}, Demonstratio Mathematica, accepted.

\bibitem{[20]} M. Z. Sarikaya and H. Ogunmez, \textit{On the weighted
Ostrowski type integral inequality for double integrals}, The Arabian
Journal for Science and Engineering (AJSE)-Mathematics, (2011) 36: 1153-1160

\bibitem{[21]} M.Z. Sarikaya and H. Ogunmez, \textit{On new inequalities via
Riemann-Liouville fractional integration}, arXiv:1005.1167v1, submitted.

\bibitem{[22]} M.Z. Sarikaya, E. Set, H. Yaldiz and N., Basak, \textit{%
Hermite -Hadamard's inequalities for fractional integrals and related
fractional inequalities}, Mathematical and Computer Modelling,
DOI:10.1016/j.mcm.2011.12.048.

\bibitem{[23]} A. Rafiq and F. Ahmad, \textit{Another weighted Ostrowski-Gr%
\"{u}ss type inequality for twice differentiable mappings}, Kragujevac
Journal of Mathematics, 31 (2008), 43-51.

\bibitem{[24]} F. Tong and L. Guan, \textit{A simple proof of the
generalized Ostrowski-Gr\"{u}ss type integral inequality,} Int. Jour. of
Math. Analysis, 2(18), (2008), 889-892.
\end{thebibliography}
\end{document}